\documentclass{article}

\usepackage{graphicx}
\usepackage{pgf}
\usepackage{amsfonts,amsmath,amssymb,mathtools}
\usepackage{url}
\usepackage[utf8]{inputenc}
\usepackage{amsthm}
\usepackage{rotating}
\usepackage{booktabs}
\usepackage[textsize=scriptsize]{todonotes}
\usepackage{hyperref}
\usepackage{subcaption}
\usepackage{authblk}
\usepackage{microtype}

\graphicspath{{./figures/}}

\everymath=\expandafter{\the\everymath\displaystyle}
\IfFileExists{scrextend.sty}{
  \usepackage[fontsize=10.000000pt]{scrextend}
}{
  \renewcommand{\normalsize}{\fontsize{10.000000}{12.000000}\selectfont}
  \normalsize
}

\makeatletter\@ifpackageloaded{underscore}{}{\usepackage[strings]{underscore}}\makeatother

\newtheorem{theorem}{Theorem}[section]
\newtheorem{lemma}[theorem]{Lemma}
\newtheorem{corollary}[theorem]{Corollary}

\theoremstyle{definition}

\newtheorem{remark}[theorem]{Remark}

\newcommand{\RR}{\mathbb{R}}
\let\SS\relax
\newcommand{\SS}{\mathbb{S}}
\newcommand{\EE}{\mathbb{E}}
\newcommand{\PP}{\mathbb{P}}
\newcommand{\dd}{\,\mathrm{d}}

\newcommand{\norm}[1]{\left\|{#1}\right\|}

\newcommand{\scp}[2]{\langle {#1},{#2}\rangle}
\newcommand{\mgeq}{\succcurlyeq}

\DeclarePairedDelimiter\p{\lparen}{\rparen}

\DeclareMathOperator*{\argmin}{argmin}
\DeclareMathOperator*{\minimize}{minimize}
\DeclareMathOperator{\diag}{diag}
\DeclareMathOperator{\Cov}{Cov}
\DeclareMathOperator{\tr}{tr}

\DeclareMathOperator{\Unif}{Unif}

\title{Why the noise model matters:\\ A performance gap in learned regularization}

\author[1]{Sebastian Banert}
\author[2]{Christoph Brauer}
\author[1]{Dirk Lorenz}
\author[1]{Lionel Tondji}

\affil[1]{Center for Industrial Mathematics, University of Bremen,\protect\\ Postfach 330440, 28334 Bremen, Germany}
\affil[2]{Institute of Lightweight Systems, German Aerospace Center,\protect\\ Ottenbecker Damm 12, 21684 Stade, Germany}

\begin{document}
\maketitle

\begin{abstract}
  This article addresses the challenge of learning effective regularizers for linear inverse problems. We analyze and compare several types of learned variational regularization against the theoretical benchmark of the optimal affine reconstruction, i.e. the best possible affine linear map for minimizing the mean squared error. It is known that this optimal reconstruction can be achieved using Tikhonov regularization, but this requires precise knowledge of the noise covariance to properly weight the data fidelity term.
  However, in many practical applications, noise statistics are unknown. We therefore investigate the performance of regularization methods learned without access to this noise information, focusing on Tikhonov, Lavrentiev, and quadratic regularization. Our theoretical analysis and numerical experiments demonstrate that for non-white noise, a performance gap emerges between these methods and the optimal affine reconstruction. Furthermore, we show that these different types of regularization yield distinct results, highlighting that the choice of regularizer structure is critical when the noise model is not explicitly learned. Our findings underscore the significant value of accurately modeling or co-learning noise statistics in data-driven regularization.
\end{abstract}



\section{Introduction}
\label{sec:intro}

In this paper we are interested in the problem of learning to regularize a linear inverse problem in a supervised fashion. The goal of regularization is to provide a map that gives good approximations to the true solutions when applied to noisy data. We are specifically interested in learning regularizers for variational regularization such as Tikhonov regularization. 

We assume that we are given a forward operator $A\in\RR^{m\times n}$ and that we also have access to samples of clean data $x^{\dag}\in\RR^{n}$. With this we can then form clean measurements $Ax$ and noisy measurements $y$ by adding noise to the clean measurements. This will generate pairs \((x^{\dag}, y)\) that can be used for training.
We may also be in the situation where we are already given pairs of (sufficiently clean) data $x^{\dag}$ and noisy measurements $y$ without having access to the underlying noise. One example is noise removal and dereverberation in audio processing.
To collect paired data for training and testing dereverberation models, researchers can use two microphones in the same room: a high-quality microphone close to the source to record clean audio $x^{\dag}$ and a lower-quality microphone further from the source that captures the reverberation as well as noise. Several sources of noise, e.g. from wind or background clutter, have unknown noise characteristics.

Learning regularization of inverse problems has been in the focus of research for some years now~\cite{argyrou2012tomographic,adler2017solving,arridge2019solving} and one particular focus is on learning variational regularization~\cite{hauptmann2018model,hammernik2018learning,lunz2018adversarial,mukherjee2020learned,mukherjee2021end}. Variational regularization sets up an objective function consisting of a discrepancy term $\mathcal{D}$ and a regularizer $\mathcal{R}$ and then one calculates the regularized solution by solving
\begin{equation*}
\minimize_{x}\quad \mathcal{D}(Ax,y) + \mathcal{R}(x)
\end{equation*}
(usually the regularized $\mathcal{R}$ is weighted by a positive regularization parameter $\alpha$ which we omit here). To learn the regularizer $\mathcal{R}$ one follows the paradigm of risk minimization (see, e.g.,~\cite{vapnik1991principles}) and considers the following bilevel optimization problem 
\begin{equation}\label{eq:bilevel-learning}
  \begin{alignedat}{3}
    \minimize_{\mathcal{R}} &\quad && 
      \EE_{x^{\dag},y}\,\ell(\hat{x}(y),x^{\dag}) \\
    \text{s.t.} &\quad && 
      \hat{x}(y) \in \argmin_{x} {\cal D}(Ax,y) + {\cal R}(x),
  \end{alignedat}
\end{equation}
where the expectation in the upper level problem is over paired data $(x,y)$ and $\ell$ is a loss function that quantifies the quality of the reconstruction. In practice, the expectation in the upper level problem is replaced by the empirical expectation, i.e., by the mean over all available pairs $(x,y)$. 
As a matter of fact, the approach~\eqref{eq:bilevel-learning} is often not applied in its plain form, since the bilevel problem is usually too hard to solve. Hence, approximate, methods are used, i.e. the method of unrolling/unfolding an algorithm that solves the lower level problems (see also the next section). 

In this work we investigate bilevel variational learning theoretically to understand its limitations. 
We are not aware of many works that provide a theoretical analysis of bilevel learning of variational methods for inverse problems.
The work~\cite{alberti2021learning} investigates standard Tikhonov regularization and assumes that the noise distribution is known.
The work~\cite{brauer2024learning} investigates both bilevel learning and unrolling for denoising by Tikhonov regularization without assuming that the noise distribution is known.
Here we also focus on regularization by Tikhonov regularization and related methods, namely Lavrentiev regularization of the normal equations and regularization with a general, not necessarily convex, quadratic regularization functional and we do not assume that the distribution of the noise is known. 
We will also provide some computational experiments to see if our theoretical findings can be observed in practice.

\subsection{State of the art}
\label{sec:sota}

The success of deep learning has spurred the development of data-driven methods for solving inverse problems, which largely fall into two distinct paradigms: (i) end-to-end networks that directly map measurements to reconstructions, often by unrolling iterative algorithms, and (ii) learned regularization methods that replace handcrafted priors within a classical variational framework. End-to-end approaches generally offer very fast reconstruction times but typically require supervised training with large sets of paired measurement and ground-truth data. In contrast, learned regularization methods often retain the interpretability and theoretical guarantees of the variational setting and can sometimes be trained on unpaired data, but their application requires solving a potentially slow and non-convex optimization problem at test time.

Within the learned regularization paradigm, several approaches have focused on parameterizing the regularizer $\mathcal{R}$ with a deep neural network. A foundational theoretical framework for this is the Network Tikhonov (NETT) approach, for which Li et al. \cite{li2020nett} established a complete convergence analysis. The primary advantage of NETT is its theoretical underpinning, providing well-posedness and convergence rate results for non-convex learned regularizers. A notable limitation, however, is its focus on analysis at the expense of a less sophisticated training scheme compared to more recent methods. A training methodology was introduced by Lunz et al. \cite{lunz2018adversarial} with Adversarial Regularizers (AR). The key advantage of AR is its flexible, unsupervised training protocol, which learns a critic network to distinguish between unregularized reconstructions and ground-truth images. The principal drawback is that the resulting non-convex regularizer leads to an iterative reconstruction process with no guarantees of convergence to a global minimum. Addressing this, Mukherjee et al. \cite{mukherjee2021end} proposed Adversarial Convex Regularizers (ACR), which constrain the regularizer network to be input-convex. The advantage of ACR is that it yields a convex variational problem, guaranteeing a unique optimal solution. The inherent trade-off, however, lies in the reduced expressive power of convex functions, which may limit reconstruction quality.

The alternative paradigm of algorithm unrolling has also proven highly effective. Hammernik et al. \cite{hammernik2018learning} introduced the Variational Network (VN), and Hauptmann et al. \cite{hauptmann2018model} proposed Deep Gradient Descent (DGD), both of which unroll an iterative scheme and learn its components end-to-end. The main advantage of these methods is their combination of model-based structure with the speed of a single forward pass at test time. A significant challenge for these methods is their reliance on strictly supervised training, which requires large corpora of paired data often unavailable in practice. Seeking to bridge these paradigms, Mukherjee et al. \cite{mukherjee2021end} developed Unrolled Adversarial Regularization (UAR), a hybrid method that adversarially co-trains a fast, unrolled network alongside a regularizer using only unpaired data. The advantage of UAR is its ability to combine the speed of end-to-end methods with the flexibility of unsupervised training. This, however, comes at the cost of increased complexity in the training pipeline, which requires carefully balancing a data-fidelity loss with an adversarial, distribution-matching loss.

A significant body of work has focused on learning regularizers with explicit convexity constraints, which guarantee a unique solution to the variational problem and allow for the use of provably convergent optimization algorithms. Mukherjee et al. \cite{mukherjee2021learning} push this concept further by proposing a method to learn a convex regularizer that also satisfies the variational source condition. The principal advantage of this approach is its theoretical foundation, as it enables the derivation of explicit convergence rates for the reconstruction. Its main limitation, however, is that the additional constraint imposed during training can lead to a slight deterioration in empirical performance compared to its less constrained counterpart. Taking a different route to provable and reliable models, Goujon et al. \cite{goujon2023neural} introduce a shallow, neural-network-based regularizer built from convex-ridge functions. The strength of this method lies in its simplicity, universality, and fast training protocol, where the regularizer is learned as a multi-step denoiser. A potential shortcoming is that the shallow architecture, while interpretable, may not possess the same expressive capacity as deeper, more complex models. Bridging the gap between convex and fully non-convex priors, Zhang and Leong \cite{zhang2025learning} propose learning a regularizer with a Difference-of-Convex (DC) structure. The key advantage here is the balance struck between flexibility and theoretical tractability; the DC formulation is more expressive than purely convex models but still allows for the use of specialized, convergent optimization algorithms like DCA. The primary trade-off is the increased complexity of the reconstruction, which requires these non-standard solvers.
Other approaches have explored alternative structural priors or have focused on different aspects of the learning problem. Alberti et al. \cite{alberti2025learning} take a unique statistical approach by modeling the signal prior as a Gaussian Mixture Model (GMM). They derive the exact Bayes estimator, which can be interpreted as a specific two-layer neural network with an attention-like mechanism. This method's main advantage is its interpretability and strong probabilistic grounding. Its scalability, however, presents a significant challenge, as the number of parameters grows prohibitively with the signal dimension and the number of mixture components, making it best suited for problems with known structured sparsity. At the other end of the architectural spectrum, Kobler et al. \cite{kobler2021total} introduce the Total Deep Variation (TDV) regularizer, a deep, multi-scale convolutional network whose training is framed as a mean-field optimal control problem. The strength of TDV lies in its empirical performance and its stability analysis with respect to both inputs and model parameters. The associated drawback is that the resulting variational problem is non-convex, meaning convergence to a global minimizer is not guaranteed, and the training itself is conceptually complex.

The remainder of this article is organized as follows: in Section
\ref{sec:setup}, we discuss the precise assumptions on the inverse
problems and the regularization methods that we consider. In Section
\ref{sec:learning-affine}, we revisit the results of
\cite{alberti2021learning} that establish the best affine
reconstruction method. In section \ref{sec:learning-noise-and-reg}, we
show that all reconstruction methods recover this best affine
regularizer when the noise weight is learned. Section
\ref{sec:learning-regularizers} is devoted to finding optimal
Lavrentiev and quadratic regularizers more explicitly, and the
implications of the theoretical results is discussed in Section
\ref{sec:discussion}. The numerical experiments in Section
\ref{sec:numerical-experiments} confirm that the discrepancies between
different regularizers can be observed in realistic scenario, provided
that the noise covariance is not too simple. Concluding remarks will
be given in \ref{sec:conclusion}.

\subsection{Notation}
\label{sec:notation}
Throughout the paper, all random variables will be defined on a fixed probability space \((\Omega, \mathcal{F}, \PP)\). The \emph{expectation} of an integrable, \(\RR^n\)-valued random variable \(x\colon \Omega \to \RR^n\) will be denoted by \(\EE\p{x} = {\textstyle \int} x\p{\omega} \,\dd\PP\p{\omega} \in \RR^n\). We shall usually abbreviate this quantity by writing \(\mu_x\). The \emph{covariance} of a square-integrable, \(\RR^n\)-valued random variable \(x\colon \Omega \to \RR^n\) is \(\Cov\p{x} = \EE\p{\p{x - \mu_x}\p{x - \mu_x}^\top} \in \RR^{n \times n}\) and will usually be denoted by \(\Sigma_x\).

We denote the identity matrix of size $n \times n$ by $I_{n}$ and drop
the subscript if no confusion arises. We write $D\mgeq 0$ to mean that $D$ is symmetric and positive semidefinite and the set of all such matrices is denoted by $\SS_{\mgeq 0}^{n}$ while we use $\SS^{n}$ for the set of symmetric matrices of size $n\times n$.
For  $D\in\SS^{n}_{\mgeq 0}$ we denote the induced inner product and norm on $\RR^{n}$ by 
\begin{align*}
\scp{x}{y}_{D} = \scp{x}{Dy},\quad \norm{x}_{D} = \sqrt{\scp{x}{x}_{D}},
\end{align*}
respectively. With slight abuse of terminology we will refer to $D$ as the \emph{metric}.

We denote the Frobenius inner product of two matrices $A,B$ of the same size and the induced norm by
\begin{align*}
  \scp{A}{B} = \tr(A^TB) = \sum_{ij}A_{ij}B_{ij},\quad
\norm{A}_{\textnormal{Fro}} = \Bigl( \sum\limits_{ij}^{}A_{ij}^2 \Bigr)^{1/2}.
\end{align*}
We will write Frobenius inner products without subscripts, but
distinguish the induced norm from other matrix norms.
An $n$-dimensional normal distribution with mean $\mu\in\RR^{n}$ and
covariance $\Sigma\in \SS^{n}_{\mgeq 0}$ is denoted by $\mathcal{N}(\mu,\Sigma)$.
With $\Unif([a,b])$ we denote the uniform distribution on the interval $[a,b]$.

\section{Problem setup}
\label{sec:setup}

Let us fix the problem setup: We assume that the true data $x^{\dag}\in\RR^{n}$ comes from a distribution with finite second moment, i.e., $x^{\dag}$ is a random vector in $\RR^{n}$ and we assume that it has mean $\EE_{x^{\dag}}(x^{\dag}) = \mu_{x^{\dag}}$ and covariance $\Cov(x^{\dag}) = \Sigma_{x^{\dag}}\in\RR^{n\times n}$. The measured data $y\in\RR^{m}$ is contaminated by random noise. We assume that the noise $\varepsilon = y - Ax$ is uncorrelated with the solution $x^{\dag}$. Moreover, we assume that the random vector $\varepsilon$ has zero mean, i.e., $\EE_{\varepsilon}(\varepsilon) = 0$, and that the covariance of the noise exists and we denote it by $\Cov_{\varepsilon}(\varepsilon)  = \Sigma_{\varepsilon}\in\RR^{m\times m}$.

As loss function in the upper level problem in \eqref{eq:bilevel-learning} we always assume the least squares function, i.e., $\ell(\hat{x},x) = \tfrac12\norm{\hat{x}-x}^{2}$.

For the lower level problem, i.e., for the regularization method, we will consider six different formulations which are, in order of increasing generality:
\begin{description}
\item[Tikhonov regularization:] The starting point for our investigation is the problem of learning quadratic Tikhonov regularization, i.e., the lower level problem is 
  \begin{align}
    \label{eq:lower-level-tikhonov-noiseweight}
    \hat{x}(y) = \argmin_{x}\tfrac12\norm{Ax-y}_{\Omega}^{2} + \tfrac12\norm{R(x-x_{0})}^{2}
  \end{align}
  where $\Omega\in\RR^{m\times m}$ is a \emph{noise weight}, i.e., a positive definite matrix, $R\in\RR^{k\times n}$ is the \emph{regularization} and $x_{0}$ models some \emph{offset}.

  In principle, $\Omega$, $R$ and $x_{0}$ can be learned, but often only the regularization $R$ and the offset $x_{0}$ are learned, see, e.g.~\cite{alberti2021learning} where the noise weight is set $\Omega=\Sigma_{\varepsilon}$, i.e., the noise is whitened.

  As many works focus on learning regularizers, we will also consider 
  \begin{align}
    \label{eq:lower-level-tikhonov}
    \hat{x}(y) = \argmin_{x}\tfrac12\norm{Ax-y}^{2} + \tfrac12\norm{R(x-x_{0})}^{2}
  \end{align}
  where the standard $\ell^{2}$-norm is used for the discrepancy term.
\item[Quadratic regularization:] Slightly more generally, we can consider to learn a quadratic regularizer that is not necessarily convex, i.e., we consider the lower level problem 
  \begin{align}
    \label{eq:lower-level-quadratic-noiseweight}
    \hat{x}(y) = \argmin_{x}\tfrac12\norm{Ax-y}_{\Omega}^{2} + \tfrac12\scp{x-x_{0}}{M(x-x_{0})}
  \end{align}
  with a square matrix $M\in\RR^{n\times n}$, which we assume without loss of generality to be symmetric. This approach is slightly more general than Tikhonov regularization as $M$ is not assumed to be positive (semi-)definite. This allows to ``regularize negatively'' in the directions with negative eigenvalues of $M$. In the context of regression it has been observed that under certain circumstances the optimal regularization parameter can be negative~\cite{kobak2020optimal} (see also~\cite{tsigler2023benign}).

  Similar to the Tikhonov case we also consider the case 
  \begin{align}
    \label{eq:lower-level-quadratic}
    \hat{x}(y) = \argmin_{x}\tfrac12\norm{Ax-y}^{2} + \tfrac12\scp{x-x_{0}}{M(x-x_{0})}
  \end{align}
  where the discrepancy uses the standard $\ell^{2}$-norm.
\item[Lavrentiev regularization:] The optimality condition for~\eqref{eq:lower-level-quadratic-noiseweight} reads as 
  \begin{align*}
    0 = A^{T}\Omega(Ax-y) + M(x-x_{0})
  \end{align*}
  which leads to the map 
  \begin{align}\label{eq:lower-level-lavrentiev-noiseweight}
    \hat{x}(y) = (A^{T}\Omega A + M)^{-1}(A^{T}\Omega y + Mx_{0}).
  \end{align}
  This is again slightly more general than the previous case since we
  do not assume that $M$ is symmetric, and we permit learning any matrix $M\in\RR^{n\times n}$ without any further assumptions. Hence, this method is in general not a variational method since it is not clear if $\hat{x}(y)$ can also be obtained as the solution of any meaningful minimization problem.

  Again, also the problem without noise weight
  \begin{align}\label{eq:lower-level-lavrentiev}
    \hat{x}(y) = (A^{T} A + M)^{-1}(A^{T} y + Mx_{0})
  \end{align}
  will be considered. This approach amounts to \emph{Lavrentiev regularization} of the normal equation $A^{T}Ax = A^{T}y$~\cite{tautenhahn2002method,hamarik2008extrapolation,semenova2010lavrentiev,lorenz2013necessary}.
\end{description}

To summarize, we collect all methods in Table~\ref{tab:regularizaton-methods}. Notably, all methods lead to an affine map $y\mapsto \hat{x}(y)$ and hence, as a baseline, we consider general affine linear maps as regularization as well.

\begin{sidewaystable}
  \centering
  \caption{Table of regularization methods considered in the paper}
  \label{tab:regularizaton-methods}
  \begin{tabular}{llp{7cm}p{2cm}}\toprule
    \textbf{Method} & \textbf{Minimization problem} & \textbf{Map} & \textbf{Parameters}\\\midrule
    Tikhonov w/ noise weight & $\tfrac12\norm{Ax-y}_{\Omega}^{2} + \tfrac12\norm{R(x-x_{0})}^{2}$ & $\hat{x}(y) = (A^{T}\Omega A + R^{T}R)^{-1}(A^{T}\Omega y + R^{T}Rx_{0})$\newline $W = (A^{T}\Omega A + R^{T}R)^{-1}A^{T}\Omega$\newline $b = (A^{T}\Omega A + R^{T}R)^{-1}R^TRx_{0}$ & $\Omega\in\SS_{\mgeq 0}^{m}$\newline $R\in\RR^{n\times n}$\newline $x_{0}\in\RR^{n}$\\\midrule
    Tikhonov w/o noise weight & $\tfrac12\norm{Ax-y}^{2} + \tfrac12\norm{R(x-x_{0})}^{2}$ & $\hat{x}(y) = (A^{T} A + R^{T}R)^{-1}(A^{T} y + R^{T}Rx_{0})$\newline $W = (A^{T} A + R^{T}R)^{-1}A^{T}$\newline $b = (A^TA + R^TR)^{-1}R^TRx_{0}$ & $R\in\RR^{n\times n}$\newline $x_{0}\in\RR^{n}$\\\midrule
    Quadratic w/ noise weight & $\tfrac12\norm{Ax-y}_{\Omega}^{2} + \tfrac12\scp{x-x_{0}}{M(x-x_{0})}$ & $\hat{x}(y) = (A^{T}\Omega A + M)^{-1}(A^{T}\Omega y + Mx_{0})$\newline $W = (A^{T}\Omega A + M)^{-1}A^{T}\Omega $\newline $b = (A^T\Omega A + M)^{-1}Mx_{0}$ & $\Omega\in\SS_{\mgeq 0}^{m}$\newline $M\in\SS^{n}$\newline $x_{0}\in\RR^{n}$\\\midrule
    Quadratic w/o noise weight & $\tfrac12\norm{Ax-y}^{2} + \tfrac12\scp{x-x_{0}}{M(x-x_{0})}$ & $\hat{x}(y) = (A^{T} A + M)^{-1}(A^{T} y + Mx_{0})$\newline $W = (A^{T} A + M)^{-1}A^{T}$\newline $b = (A^TA + M)^{-1}Mx_{0}$ & $M\in\SS^{n}$\newline $x_{0}\in\RR^{n}$\\\midrule
    Lavrentiev w/ noise weight & - & $\hat{x}(y) = (A^{T}\Omega A + M)^{-1}(A^{T}\Omega y + Mx_{0})$\newline $W = (A^{T}\Omega A + M)^{-1}A^{T}\Omega$\newline $b = (A^T\Omega A + M)^{-1}Mx_{0}$ & $\Omega\in\SS_{\mgeq 0}^{m}$\newline $M\in\RR^{n\times n}$\newline $x_{0}\in\RR^{n}$\\\midrule
    Lavrentiev w/o noise weight & - & $\hat{x}(y) = (A^{T} A + M)^{-1}(A^{T} y + Mx_{0})$\newline  $W = (A^{T} A + M)^{-1}A^{T}$\newline $b = (A^TA + M)^{-1}Mx_{0}$ & $M\in\RR^{n\times n}$\newline $x_{0}\in\RR^{n}$\\\bottomrule
  \end{tabular}
\end{sidewaystable}

\begin{description}
\item[Affine regularization:] As a generalization of all six methods, we will consider
  \begin{align}
    \label{eq:lower-level-affine-linear}
    \hat{x}(y) = Wy + b
  \end{align}
  with $W\in\RR^{n\times m}$ and $b\in\RR^{n}$.
\end{description}

Since we choose  $\ell(\hat{x},x) = \norm{\hat{x}-x}^{2}$ as upper level loss, we get as problem~\eqref{eq:bilevel-learning} to learn how to regularize the inverse problem $Ax=y$
\begin{equation}\label{eq:learning-affine-maps}
  \begin{alignedat}{3}
    \min_{W,b} &\quad && \EE_{x^{\dag},\varepsilon}\,\norm{\hat{x}(y)-x^{\dag}}^{2} \\
    \text{s.t.} &\quad && \hat{x}(y) = Wy + b,\quad y = Ax^{\dag} + \varepsilon
  \end{alignedat}
\end{equation}
where the matrix $W$ for the different methods and the respective $b$ can be found in Table~\ref{tab:regularizaton-methods}.

In general we denote the \emph{risk} of a method $\hat{x}_{\theta}$ that depends on parameters $\theta$ by 
\begin{align*}
\mathcal{R}(\theta) := \EE_{x^{\dag},\varepsilon}\norm{\hat{x}_{\theta}(y)-x^{\dag}}^{2}
\end{align*}
For the cases considered in this paper the parameters $\theta$ can be found in the last column of Table~\ref{tab:regularizaton-methods}. The optimal risk for specific method is denoted by
\begin{align*}
  \mathcal{R}_{\textnormal{Method}} = \inf_{\theta}\ \EE_{x^{\dag},\varepsilon}\norm{\hat{x}_{\theta}^{\textnormal{Method}}(y)-x^{\dag}}^{2} .
\end{align*}
In this work we consider the following optimal risks and respective parameterized regularization methods:
\bigskip

\renewcommand{\arraystretch}{1.7}
\begin{tabular}{r @{\hspace{0.2em}} l @{\hspace{2em}} r @{\hspace{0.3em}} l}
$\mathcal{R}_{\textnormal{Aff}}$ &: & $\hat{x}_{\theta}^{\textnormal{Aff}}(y)$ & $= Wy  + b$\\
$\mathcal{R}_{\textnormal{Tikh}(\Omega)}$ &: & $\hat{x}_{\theta}^{\textnormal{Tikh}(\Omega)}(y)$ & $= \argmin_{x}\tfrac12\norm{Ax-y}_{\Omega}^{2} + \tfrac12\norm{R(x-x_{0})}^{2}$\\
$\mathcal{R}_{\textnormal{Tikh}}$ &: & $\hat{x}_{\theta}^{\textnormal{Tikh}}(y)$ & $= \argmin_{x}\tfrac12\norm{Ax-y}^{2}+\tfrac12\norm{R(x-x_{0})}^{2}$\\
$\mathcal{R}_{\textnormal{Quad}(\Omega)}$ &: & $\hat{x}_{\theta}^{\textnormal{Quad}(\Omega)}(y)$ & $= \argmin_{x}\tfrac12\norm{Ax-y}_{\Omega}^{2}+\tfrac12\scp{x-x_{0}}{M(x-x_{0})}$\\
$\mathcal{R}_{\textnormal{Quad}}$ &: & $\hat{x}_{\theta}^{\textnormal{Quad}}(y)$ & $= \argmin_{x}\tfrac12\norm{Ax-y}^{2}+\tfrac12\scp{x-x_{0}}{M(x-x_{0})}$\\
$\mathcal{R}_{\textnormal{Lav}(\Omega)}$ &: & $\hat{x}_{\theta}^{\textnormal{Lav}(\Omega)}(y)$ & $= (A^{T}\Omega A + M)^{-1}(A^{T}\Omega y + Mx_{0})$\\
$\mathcal{R}_{\textnormal{Lav}}$ &: & $\hat{x}_{\theta}^{\textnormal{Lav}}(y)$ & $= (A^{T}A + M)^{-1}(A^{T} y + Mx_{0})$   
\end{tabular}
\bigskip

By construction we immediately conclude the following order of these optimal risks
\begin{align*}
  \mathcal{R}_{\textnormal{Aff}}\leq \left\{
  \begin{array}{l}
    \mathcal{R}_{\textnormal{Lav}}\leq \mathcal{R}_{\textnormal{Quad}} \leq \mathcal{R}_{\textnormal{Tikh}}\\
    \mathcal{R}_{\textnormal{Lav}(\Omega)} \leq \mathcal{R}_{\textnormal{Quad}(\Omega)} \leq \mathcal{R}_{\textnormal{Tikh}(\Omega)}.
  \end{array}
  \right. 
\end{align*}
The smaller the risk of a method, the better its performance. In the following we aim to analyze if there are performance gaps between the methods we outlined above and if they can be ordered at all.

\section{Learning the best affine reconstruction method}
\label{sec:learning-affine}

We establish the baseline and collect results on the best affine linear reconstruction map $y\mapsto Wy + b$. Most results in this section can also be found elsewhere in the literature, but we include them and their derivation for the sake of completeness.

We begin our analysis with the following result on the expression of the risk for such maps:

Given a matrix \(A \in \RR^{m\times n}\), an \(\RR^n\)-valued random variable \(x^\dagger\) and an \(\RR^m\)-valued random variable \(\varepsilon\), we define the \emph{risk} of an affine mapping \(y \mapsto Wy + b\), where \(W \in \RR^{n\times m}\) and \(b \in \RR^n\), as
\[
  \mathcal{R}(W,b) = \EE_{x^{\dag},\varepsilon}\norm{W(Ax^{\dag}+\varepsilon) + b - x^{\dag}}^{2} \in \RR \cup \{\infty\}.
\]
\begin{lemma}\label{lem:loss-decomposition}
  Let \(x^\dagger\) and \(\varepsilon\) be uncorrelated, square-integrable random variables with \(\EE\p{x^{\dag}} = \mu_{x^{\dag}}\), \(\Cov\p{x^{\dag}} = \Sigma_{^{\dag}}\in\SS_{\geq 0}^{n}\), \(\EE\p{\varepsilon} = 0\), and \(\Cov\p{\varepsilon} = \Sigma_\varepsilon\in\SS_{\geq 0}^{m}\). Then, for any \(W \in \RR^{n\times m}\) and \(b \in \RR^n\), the risk \(\mathcal{R}(W, b)\) is finite and given by
  \begin{multline}\label{eq:pop-risk}
    \begin{split}
      &\mathcal{R}(W, b) \\
      &= \scp{(WA-I)\Sigma_{x^{\dag}}}{WA-I} + \scp{W\Sigma_{\varepsilon}}{W}  + \norm{(WA-I)\mu_{x^{\dag}} + b}^{2}\\
      &= \scp{WA\Sigma_{x^{\dag}}A^{T}}{W} -
        2\scp{\Sigma_{x^{\dag}}A^{T}}{W} + \tr(\Sigma_{x^{\dag}}) \\
      &\hspace{3cm} + \scp{W\Sigma_{\varepsilon}}{W} + \norm{(WA-I)\mu_{x^{\dag}} + b}^{2}.
    \end{split}
  \end{multline}
\end{lemma}
\begin{proof}
  We add and subtract $(WA-I)\mu_{x^{\dagger}}$ in the risk
  \begin{multline*}
    \EE_{x^{\dag},\varepsilon}\norm{W(Ax^{\dag}+\varepsilon) + b -x^{\dag}}^{2}\\
    = \EE_{x^{\dag},\varepsilon}\norm{(WA-I)(x^{\dag}-\mu_{x^{\dag}}) + W\varepsilon + (WA-I)\mu_{x^{\dag}} + b}^{2},
  \end{multline*}
  expand the square, and use that $x^{\dag}$ and $\varepsilon$ are
  uncorrelated
  and that $\EE_{\varepsilon} \varepsilon = 0$ and get
  \begin{align*}
    \EE_{x^{\dag},\varepsilon}\norm{WA(x^{\dag}+\varepsilon) + b -x^{\dag}}^{2} & = \EE_{x^{\dag}}\norm{(WA-I)(x^{\dag}-\mu_{x^{\dag}})}^{2} + \EE_{\varepsilon}\norm{W\varepsilon}^{2}\\
    & \qquad + \norm{(WA-I)\mu_{x^{\dag}} + b}^{2}.
  \end{align*}
  Finally we use two instances the identity
  \[
  \EE_z \norm{L (z - \EE_z(z))}^{2} = \scp{L \Cov(z)}{L},
  \]
  which is valid for any random variable \(z\) and matrix \(L\) with
  compatible sizes, and obtain
  \begin{align*}
    &\EE_{x^{\dag}}\norm{(WA-I)(x^{\dag}-\mu_{x^{\dag}})}^{2} + \EE_{\varepsilon}\norm{W\varepsilon}^{2}\\
         & = \scp{(WA-I)\Sigma_{x^{\dag}}}{WA-I} + \scp{W\Sigma_{\varepsilon}}{W}\\
         & = \scp{WA\Sigma_{x^{\dag}}A^{T}}{W} -
           2\scp{\Sigma_{x^{\dag}}A^{T}}{W} + \tr(\Sigma_{x^{\dag}}) +
           \scp{W\Sigma_{\varepsilon}}{W}. \qedhere
  \end{align*}
\end{proof}

\begin{remark}
  Lemma \ref{lem:loss-decomposition} shows that the risks of affine linear methods
  decompose naturally into three parts: The \emph{variance} term
  $\scp{W\Sigma_{\varepsilon}}{W}$ that occurs through the noise, the
  \emph{operator bias} term $\scp{(WA-I)\Sigma_{x^{\dag}}}{WA-I}$ that is due to the approximation of the inversion process, and the
  \emph{offset bias} $\norm{(WA-I)\mu_{x^{\dag}} + b}^{2}$ that is due to the choice of offset $b$ in the  reconstruction.

  Note that the operator bias and the variance are always non-negative as we could interpret the expressions as weighted Frobenius inner products.
\end{remark}

The above result allows to decouple the minimization in the upper
level problems with respect to the offset $b$ in the affine linear
map. We immediately see that the optimal offset $b$ (provided a given $W$) is $b = (I-WA)\mu_{x^{\dag}}$. For further use, we formulate this as a corollary:

\begin{corollary}[Optimal affine offset]\label{cor:opt-offset}
  Let $W\in \RR^{n\times m}$.
  The solution \(b^{*} \in \RR^n\) of the problem to minimize the
  risk of the affine reconstruction map $y \mapsto W y +
  b$, i.e., of
  \begin{align*}
    \minimize_{b} \EE_{x^{\dag},\varepsilon}\norm{W(Ax^{\dag} + \varepsilon) + b-x^{\dag}}^{2}
  \end{align*}
  is given by
  \begin{align*}
    b^{*} = (I-WA)\mu_{x^{\dag}}\ .
  \end{align*}
\end{corollary}

Moreover, by taking the derivative of the right hand side in \eqref{eq:pop-risk} with respect to \(W\) we can easily rederive the formula for the optimal affine linear reconstruction map, namely the well known Linearized Minimum Mean Square Error (LMMSE) estimator~\cite[Theorem 12.1]{kay1993fundamentals} which has also been rederived in~\cite[Theorem 3.1]{alberti2021learning}:

\begin{corollary}[LMMSE estimation]\label{cor:llmmse}
  Let $\hat{x} = W(Ax^{\dag}+\varepsilon) + b$ with $W\in\RR^{n\times m}$ and $b\in\RR^{n}$. It holds that the problem
  \begin{align*}
    \min_{W,b} \EE_{x^{\dag},\varepsilon}\norm{\hat{x}-x^{\dag}}^{2}
  \end{align*}
  is solved by  
  \begin{align*}
    W^{*}  = \Sigma_{x^{\dag}}A^{T}(A\Sigma_{x^{\dag}}A^{T} + \Sigma_{\varepsilon})^{-1},\qquad\text{and}\qquad b^{*}  = (I-W^{*}A)\mu_{x^{\dagger}}\ .
  \end{align*}    
\end{corollary}

\begin{proof}
  The optimality of $b^{*}$ is clear by Corollary~\ref{cor:opt-offset}.
  Hence we have from Lemma~\ref{lem:loss-decomposition}  
\begin{align*}
\mathcal{R}(W,b^{*}) = \scp{WA\Sigma_{x^{\dag}}A^{T}}{W} - 2\scp{\Sigma_{x^{\dag}}A^{T}}{W} + \tr(\Sigma_{x^{\dag}}) + \scp{W\Sigma_{\varepsilon}}{W}.
\end{align*}
  To calculate the optimal $W$ we take the gradient of this with respect to $W$ and get the optimality condition
  \begin{align*}
    2 WA\Sigma_{x^{\dag}}A^{T} - 2\Sigma_{x^{\dag}}A^{T} + 2W\Sigma_{\varepsilon} = 0,
  \end{align*}
  and this is solved by 
  \begin{equation*}
    W = \Sigma_{x^{\dag}}A^{T}(A\Sigma_{x^{\dag}}A^{T} + \Sigma_{\varepsilon})^{-1}.\qedhere
  \end{equation*} 
\end{proof}

\section{Optimal regularization when the noise model is learned}
\label{sec:learning-noise-and-reg}

Now we start to analyze the learning problems where we learn noise weights $\Omega$, regularizers $R$ or $M$, respectively, and offsets $x_{0}$. We begin with the least general case, namely Tikhonov regularization with noise weight:

\begin{theorem}[Optimal Tikhonov regularization and noise weight]
  The bilevel learning problem
  \begin{align*}
    \minimize_{\Omega,R,x_{0}}& &&\EE_{x^{\dag},\varepsilon}\norm{\hat{x} - x^{\dag}}^{2}\\
    \mathrm{s.t.}& &&\hat{x} = \argmin_{x}\tfrac12\norm{Ax-y}_{\Omega}^{2} + \tfrac12\norm{R(x-x_{0})}^{2}
  \end{align*}
  is solved by any $R$, $\Omega$ and $x_{0}$ with
\begin{align*}
\Omega = \Sigma_{\varepsilon}^{-1},\quad\quad R^{T}R = \Sigma_{x^{\dag}}^{-1},\quad\text{and}\quad x_{0} = \mu_{x^{\dag}}.
\end{align*}
Especially, the respective affine map is equal to the LMMSE map from Corollary~\ref{cor:llmmse}.
\end{theorem}

\begin{proof}
  The affine map that solves the lower level problem is 
\begin{align*}
\hat{x} = (A^{T}\Omega A + R^TR)^{-1}(A^{T}\Omega y + R^{T}Rx_{0}).
\end{align*}
We show that the choices for $\Omega$, $R$ and $x_{0}$ above turn this affine map into the LMMSE estimator from Corollary~\ref{cor:llmmse}:
It holds that $\hat{x} = Wy + b$ with 
\begin{align*}
W = (A^{T}\Omega A + R^TR)^{-1}A^{T}\Omega,\qquad b = (A^{T}\Omega A + R^TR)^{-1}R^{T}Rx_{0}.
\end{align*}
Setting $W = W^{*}$ (with $W^{*}$ from Corollary~\ref{cor:llmmse}) and moving the inverses to the respective other sides shows that $W=W^{*}$ if 
\begin{align*}
R^{T}R \Sigma_{x^{\dag}} = I_{n},\quad\text{and}\quad \Omega\Sigma_{\varepsilon} = I_{m}.
\end{align*}
Especially we conclude that $W^{*} = (A^{T}\Sigma_{\varepsilon}^{-1}A + \Sigma_{x^{\dag}}^{-1})^{-1}A^{T}\Sigma_{\varepsilon}^{-1}$. With this we compute 
\begin{align*}
  b^{*} = \mu_{x^{\dag}} - (A^{T}\Sigma_{\varepsilon}^{-1}A + \Sigma_{x^{\dag}}^{-1})^{-1}A^{T}\Sigma_{\varepsilon}^{-1}A\mu_{x^{\dag}} = (A^T\Sigma_{\varepsilon}^{-1}A + \Sigma_{x^{\dag}}^{-1})^{-1}\Sigma_{x^{\dag}}^{-1}\mu_{x^{\dag}}.
\end{align*}
Equating this with the offset $(A^{T}\Sigma_{\varepsilon}^{-1}A + \Sigma_{x^{\dag}}^{-1})^{-1}\Sigma_{x^{\dag}}^{-1}x_{0}$ of the optimal map from Tikhonov regularization we get $\mu_{x^{\dag}} = x_{0}$.
\end{proof}

The above result is basically already contained in~\cite{alberti2021learning}, but there the authors had fixed $\Omega = \Sigma_{\varepsilon}^{-1}$ already and only showed that learning $R$ and $x_{0}$ leads to the LMMSE estimator.

Since Tikhonov regularization is more special than  quadratic and Lavrentiev regularization (including the noise weights), we have also shown that these methods can also achieve to learn the LMMSE estimate which is the best possible affine linear map.

In other words, we have derived that the respective optimal risks are ordered as 
\begin{align*}
\mathcal{R}_{\text{Aff}}= \mathcal{R}_{\text{Lav}(\Omega)} = \mathcal{R}_{\text{Quad}(\Omega)} = \mathcal{R}_{\text{Tikh}(\Omega)} \leq \mathcal{R}_{\text{Lav}}\leq \mathcal{R}_{\text{Quad}} \leq \mathcal{R}_{\text{Tikh}}.
\end{align*}
In the following we investigate the remaining inequalities.

\section{Learning regularizers without a noise model}
\label{sec:learning-regularizers}

In this section we start with the most general method: Learning the Lavrentiev regularization of the normal equations (and the offset). The problem we aim to solve is
\begin{align*}
  \mathcal{R}_{\textnormal{Lav}} = \min_{M,x_{0}}& \quad \EE_{x^{\dag},\varepsilon}\norm{\hat{x}(Ax^{\dag} + \varepsilon) - x^{\dag}}^{2}\\
  \text{s.t.}\ & \hat{x}(Ax^{\dag}+\varepsilon) = (A^{T}A + M)^{-1}(A^{T}(Ax^{\dag}+\varepsilon) + Mx_{0})
\end{align*}

\begin{theorem}[Optimal Lavrentiev regularization]\label{thm:optimal-lavrentiev}
  Let the matrices $A^{T}A$, $\Sigma_{x^{\dag}}$ and $\Sigma_{\varepsilon}$ be invertible. The bilevel learning problem
  \begin{align*}
    \minimize_{M,x_{0}}\ & \EE_{x,\varepsilon}\norm{\hat{x} - x^{\dag}}^{2}\\
    \text{s.t.}& \qquad \hat{x} = (A^{T}A + M)^{-1}(A^{T}(Ax^{\dag}+\varepsilon) + Mx_{0})
  \end{align*}
  is solved by all $M$ and $x_{0}$ with
  \begin{align*}
    x_{0} \in \mu_{x^{\dag}} + \ker(M),\qquad M = A^{T}\Sigma_{\varepsilon}A(A^{T}A)^{-1}\Sigma_{x^{\dag}}^{-1}.
  \end{align*}
\end{theorem}
\begin{proof}
  We introduce additional variables $W$ and $b$ and rewrite the bilevel problem as
  \begin{equation}\label{eq:opt-lavrentiev-problem}
  \begin{aligned}
    \minimize_{W,b,M,x_{0}}\ & \EE_{x,\varepsilon}\norm{W(Ax^{\dag}+\varepsilon)+b - x^{\dag}}^{2}\\
    \text{s.t.}\ & W = (A^TA + M)^{-1}A^{T},\quad b = (A^{T}A + M)^{-1}Mx_{0}.
  \end{aligned}
  \end{equation}
  Using the risk decomposition from Lemma~\ref{lem:loss-decomposition} and the definitions of $W$ and $b$
  we see that to minimize in $x_0$, we need to minimize
  \begin{align*}
    \norm{(WA-I)\mu_{x^{\dag}} + b}^{2} &= \norm{(A^TA + M)^{-1}A^{T}A\mu_{x^{\dag}} - \mu  + (A^{T}A + M)^{-1}Mx_{0}} \\
                                   &= \norm{(A^TA + M)^{-1} M\big(x_{0} - \mu_{x^{\dag}} \big)} 
  \end{align*}
  over $x_{0}$, and this is solved by any $x_0 \in \mu_{x^{\dag}}+ \ker(M)$.

  Now we are going to solve \eqref{eq:opt-lavrentiev-problem} for $W$ and $M$. The remaining problem is (where we again used the risk decomposition from Lemma~\ref{lem:loss-decomposition} and also inverted the constraints)
  \begin{align*}
    \minimize_{W,M}\ & \scp{WA\Sigma_{x^{\dag}}A^{T}}{W} - 2\scp{\Sigma_{x^{\dag}}A^{T}}{W} + \tr(\Sigma_{x^{\dag}}) + \scp{W\Sigma_{\varepsilon}}{W}\\
    \text{s.t.}\ & (A^TA + M)W = A^{T}.
  \end{align*}
  The Lagrangian of this nonlinear optimization problems is given by
  (introducing a factor of \(\frac{1}{2}\))
  \begin{align*}
    \mathcal{L}(W,M,\Lambda) & = \tfrac{1}{2}\scp{WA\Sigma_{x^{\dag}}A^{T}}{W} - \scp{\Sigma_{x^{\dag}}A^{T}}{W}  +\tfrac{1}{2} \scp{W\Sigma_{\varepsilon}}{W}\\
             & \qquad+ \scp{(A^{T}A+M)W - A^{T}}{\Lambda},
  \end{align*}
  and thus, the optimality conditions are
  \begin{align}
    \nabla_{W}\mathcal{L} & = WA\Sigma_{x^{\dag}}A^{T} - \Sigma_{x^{\dag}}A^{T} + W\Sigma_{\varepsilon} + (A^{T}A+M)^{T}\Lambda = 0, \label{eq:eqt1}\\
    \nabla_{M}\mathcal{L} & =\Lambda W^{T} = 0, \label{eq:eqt2}\\
    \nabla_{\Lambda}\mathcal{L} & = (A^{T}A+M)W-A^{T} = 0. \label{eq:eqt3} 
  \end{align}
  We transpose~\eqref{eq:eqt2} and multiply it by $A^TA + M$ from the left to get $(A^{T}A + M)W\Lambda^{T} = 0$. From~\eqref{eq:eqt3} we conclude that $A^{T}\Lambda ^{T}= 0$.
  We multiply~\eqref{eq:eqt1} from the right by $A$ and obtain 
  \begin{align*}
  W(A\Sigma_{x^{\dag}}A + \Sigma_{\varepsilon})A - \Sigma_{x^{\dag}}A^{T}A + (A^{T}A + M)^{T}\Lambda A = 0.
  \end{align*}
  The last summand on the left hand side vanishes and we multiply from the left with $A^TA + M$ and get (using~\eqref{eq:eqt3} again) 
  \begin{align*}
  A^{T}(A\Sigma_{x^{\dag}}A^T + \Sigma_{\varepsilon})A = (A^{T}A + M)\Sigma_{x^{\dag}}A^TA.
  \end{align*}
  We can solve this for $M$ and get 
  \begin{align*}
    M & = A^{T}(A\Sigma_{x^{\dag}}A^T + \Sigma_{\varepsilon})A(A^{T}A)^{-1}\Sigma_{x^{\dag}}^{-1} - A^TA\\
    & = A^{T}A + A^{T}\Sigma_{\varepsilon}A(A^{T}A)^{-1}\Sigma_{x^{\dag}}^{-1} - A^{T}A = A^{T}\Sigma_{\varepsilon}A(A^{T}A)^{-1}\Sigma_{x^{\dag}}^{-1} .
  \end{align*}
  It remains to show that $A^{T}A + M$ is invertible. We rewrite 
  \begin{align*}
    A^TA + M & = A^{T}A + A^{T}\Sigma_{\varepsilon}A(A^{T}A)^{-1}\Sigma_{x^{\dag}}^{-1}  \\
    & = (A^{T}A\Sigma_{x^{\dag}}A^{T}A + A^{T}\Sigma_{\varepsilon}A)(A^{T}A)^{-1}\Sigma_{x^{\dag}}^{-1}
  \end{align*}
  and observe that our assumptions imply that all three factors on the right are invertible matrices.
\end{proof}

\begin{remark}
  The map $W = (A^{T}A + M)^{-1}A^{T}$ that corresponds to the optimal Lavrentiev regularizer $M =A^{T}\Sigma_{\varepsilon}A(A^{T}A)^{-1}\Sigma_{x^{\dag}}^{-1}$ can be rearranged into the form
  \begin{multline*}
    (A^TA + A^{T}\Sigma_{\varepsilon}A(A^{T}A)^{-1}\Sigma_{x^{\dag}}^{-1})^{-1} A^T \\= \Bigg(\Big(A^TA\Sigma_{x^{\dag}}A^TA + A^{T}\Sigma_{\varepsilon}A\Big)(A^{T}A)^{-1}\Sigma_{x^{\dag}}^{-1}  \Bigg)^{-1} A^T \\
    = \Sigma_{x^{\dag}} A^{T}A \Bigg(A^TA\Sigma_{x^{\dag}}A^TA + A^{T}\Sigma_{\varepsilon}A  \Bigg)^{-1} A^T.
  \end{multline*}
\end{remark}

\begin{theorem}[Optimal quadratic regularization]\label{thm:optimal-quadratic}
  Assume that the matrix $B:= A^{T}(A\Sigma_{x^{\dag}}A^{T} +
  \Sigma_{\varepsilon})A$ be invertible. Then, the bilevel learning problem
  \begin{align*}
    \minimize_{M,x_{0}} &\ \EE_{x^{\dag},\varepsilon} \norm{\hat{x} - x^{\dag}}^{2}\\
    \text{s.t.} &\qquad \hat{x} = \argmin_{x} \norm{Ax - A(x^{\dag}+\varepsilon)}^{2} + \scp{M(x-x_{0})}{x-x_{0}}
  \end{align*}
  has the following solution: Let $N$ be the unique symmetric solution of the Lyapunov equation 
  \begin{align*}
    A^{T}A\Sigma_{x^{\dag}} + \Sigma_{x^{\dag}}A^{T}A = NB + BN.
  \end{align*}
  If $N$ is invertible, then the optimal $M$ and $x_{0}$ are given by
  \begin{align*}
    x_{0} \in \mu_{x^{\dag}} + \ker(M)\quad\text{and}\quad\quad M = N^{-1} - A^{T}A.
  \end{align*}
\end{theorem}
\begin{proof}
  In the bilevel problem we assume without loss of generality that the $M$ is symmetric.
  We can replace the lower level problem with the optimality condition
  and add the symmetry of $M$ as an additional constraint to obtain the non-linear problem
  \begin{align*}
    \min_{M,x_{0}} &\ \EE_{x^{\dag},\varepsilon} \norm{\hat{x} - x^{\dag}}^{2}\\
    \text{s.t.}& \qquad \hat{x} = (A^{T}A + M)^{-1}(A^{T}(Ax^{\dag}+\varepsilon) + Mx_{0})\quad\text{and}&
                 M &= M^{T}.
  \end{align*}
  Similarly to the proof of Theorem~\ref{thm:optimal-lavrentiev} we
  use the risk decomposition from Lemma~\ref{lem:loss-decomposition} to observe that $x_{0} = \mu_{x^{\dag}} + \ker(M)$ solves the minimization over $x_{0}$.
  To find the optimal $M$ we again introduce the variable $W$ and arrive at
  \begin{align*}
    \min_{M,W} & \ \scp{WA\Sigma_{x^{\dag}}A^{T}}{W} - 2\scp{\Sigma_{x^{\dag}}A^{T}}{W} + \tr(\Sigma_{x^{\dag}}) + \scp{W\Sigma_{\varepsilon}}{W}\\
    \text{s.t.}& \qquad (A^{T}A + M)W = A^{T},\quad\text{and}\quad M = M^{T}.
  \end{align*}
  The Lagrangian of this is (introducing a factor of $\tfrac12$) 
  \begin{align*}
    \mathcal{L}(W,M,\Theta,\Lambda) & = \tfrac12 \scp{WA\Sigma_{x^{\dag}}A^{T}}{W} - 2\scp{\Sigma_{x^{\dag}}A^{T}}{W} + \tr(\Sigma_{x^{\dag}}) + \scp{W\Sigma_{\varepsilon}}{W}\\
               & \qquad + \scp{(A^{T}A+M)W - A^{T}}{\Lambda} + \scp{M - M^{T}}{\Theta}
  \end{align*}
  and the condition that its derivatives have to vanish are 
  \begin{align}
    0=\nabla_{W}\mathcal{L} & = WA\Sigma_{x^{\dag}}A^{T} - \Sigma_{x^{\dag}}A^{T} + W\Sigma_{\varepsilon} + (A^{T}A+M)^{T}\Lambda \label{eq:eq1}\\
    0=\nabla_{M}\mathcal{L} & =M - M^{T} + \Lambda W^{T}     \label{eq:eq2}\\
    0=\nabla_{\Lambda}\mathcal{L} & = (A^{T}A+M)W-A^{T}     \label{eq:eq3} \\
    0=\nabla_{\Theta}\mathcal{L} & = M - M^{T}.\label{eq:eq4}
  \end{align}
  It follows that
  \begin{align}
    \eqref{eq:eq4} &\Longrightarrow  M = M^{T} \label{eq:eq4-2}\\
    \eqref{eq:eq3} &\Longrightarrow W = (A^{T}A+M)^{-1}A^{T} \label{eq:eq3-2}\\
    \eqref{eq:eq2} &\Longrightarrow \Lambda W^{T}+ W \Lambda^{T}=0\label{eq:eq2-2}
  \end{align}  
  From~\eqref{eq:eq1} we get
  \begin{align}
    W(A\Sigma_{x^{\dag}}A^{T} + \Sigma_{\varepsilon}) - \Sigma_{x^{\dag}}A^{T} + (A^{T}A+M)\Lambda  = 0   \label{eq:eq5}.
  \end{align}
  and \eqref{eq:eq3-2} and \eqref{eq:eq2-2} give us
  \begin{align*}
    \Lambda A(A^{T}A + M)^{-1} + (A^{T}A + M)^{-1}A^{T}\Lambda^{T} = 0.
  \end{align*}
  We cancel the inverses by multiplying from the left and right by $(A^{T}A + M)$ to arrive at
  \begin{equation}
    (A^{T}A+M)\Lambda A + A^{T} \Lambda^{T}(A^{T}A+M) = 0.      \label{eq:eq6}
  \end{equation}
  From~\eqref{eq:eq5}, it follows that :
  \begin{align*}
    0 = (A^{T}A+M)^{-1}A^{T}(A\Sigma_{x^{\dag}}A^{T} + \Sigma_{\varepsilon}) - \Sigma_{x^{\dag}}A^{T} + (A^{T}A+M)\Lambda 
  \end{align*}
  and multiplying from the right by $A$ gives us
  \begin{align*}
    0 = (A^{T}A+M)^{-1}A^{T}(A\Sigma_{x^{\dag}}A^{T} + \Sigma_{\varepsilon})A - \Sigma_{x^{\dag}}A^{T}A + (A^{T}A+M)\Lambda A.
  \end{align*}
  We add the transpose of the equality to itself and 
  use~\eqref{eq:eq6} to observe that two terms cancel each other and arrive at
  \begin{align*}
    A^{T} A\Sigma_{x^{\dag}} + \Sigma_{x^{\dag}}A^{T}A  & = (A^{T}A+M)^{-1}A^{T}(A\Sigma_{x^{\dag}}A^{T}+ \Sigma_{\varepsilon})A\\
    & \qquad+ A^{T} (A\Sigma_{x^{\dag}}A^{T} + \Sigma_{\varepsilon})A(A^{T}A+M)^{-1}.
  \end{align*}
  With $N = A^{T}A + M$ this is exactly the stated Lyapunov equation, and the existence of a unique solution follows from Sylvester's theorem \cite[Thm. VII.2.1]{bhatia2013matrix} since $B$ is invertible and this can be seen to be symmetric from the explicit formula
  \begin{align*}
    N = \int\limits_{0}^{\infty}e^{-tB}(A^{T} A\Sigma_{x^{\dag}} + \Sigma_{x^{\dag}}A^{T}A)e^{-tB}\dd t,
  \end{align*}
  cf.~\cite[Section 5.3]{lancaster1995algebraic}.
\end{proof}

\begin{remark}\label{rem:formula-lyapunov}
  In this case we also have a formula for the solution $M$ that is more useful for numerical implementation: We diagonalize $B = U\diag(\beta_{1},\dots,\beta_{n})U^{T}$ with $\beta_{i}>0$ for $i=1,\dots,n$ and $U = [u_{1},\dots, u_{n}]$ and set $D := A^{T}A\Sigma_{x^{\dag}} + \Sigma_{x^{\dag}}A^{T}A$ and get the solution as
  \begin{align}\label{eq:best-M-quadrativ}
    M = N^{-1} - A^{T}A\quad\text{with}\quad N = U \left( \frac{\scp{u_{i}}{Du_{j}}}{\beta_{i} + \beta_{j}} \right)_{i,j}U^{T}.
  \end{align}
  In Theorem~\ref{thm:optimal-quadratic} we had to assume that $N$ is invertible and we do not know if this is always fulfilled. In our numerical experiments this was always the case. 

  From the last equality of the proof of Theorem~\ref{thm:optimal-quadratic} we see a necessary condition for $N$ to be positive definite, namely when $A^{T}A\Sigma_{x^{\dag}} + \Sigma_{x^{\dag}}A^{T}A$ is positive definite.
  The resulting $M$ is not guaranteed to be positive semidefinite, even in this case.
  It may seem counterintuitive at first that an indefinite matrix with some negative eigenvalues is suitable for regularization.
  However, this phenomenon has been observed even stronger in the case of ridge regression where sometimes even the regularization parameter can be negative~\cite{kobak2020optimal}.
\end{remark}

We have derived optimal regularization methods for all models that we proposed in Section~\ref{sec:setup} except for Tikhonov regularization without noise weight and this remains an open problem. In Secion~\ref{sec:numerical-experiments} we will do a numerical experiment in which we learn the respective map $\hat{x}_{\theta}^{\textnormal{Tikh}}$ by gradient descent to compare the performance with the other methods.

\section{Discussion}
\label{sec:discussion}

Let us recall the main results from Section~\ref{sec:learning-regularizers}: The optimal matrix $M$ for Lavrentiev regularization of the normal equations is 
\begin{align*}
M = A^{T}\Sigma_{\varepsilon}A(A^{T}A)^{-1}\Sigma_{x^{\dag}}^{-1},
\end{align*}
and the optimal $M$ for the quadratic regularization is of the form
\begin{align*}
M = N^{-1} - A^{T}A
\end{align*}
where $N$ is the solution of a Lyapunov equation.

From this we can already suspect a few facts:
\begin{itemize}
\item Since the $M$ from the Lavrentiev regularization is in general not symmetric, the optimal quadratic regularization is in general worse in the sense that $\mathcal{R}_{\textnormal{Lav}}<\mathcal{R}_{\textnormal{Quad}}$, i.e. there is a performance gap between these methods.
\item The optimal $M$ for quadratic regularization does not seem to be positive definite in general (it only is if $N^{-1}\mgeq A^{T}A$). This implies that it is to be expected that there a performance gap between Tikhonov and quadratic regularization in the sense that $\mathcal{R}_{\textnormal{Quad}}<\mathcal{R}_{\textnormal{Tikh}}$.
\end{itemize}

Here is a result that shows when the optimal Lavrentiev regularization (without noise weight) is in fact as good LMMSE:
\begin{theorem} \label{thm:lavrentiev_equals_lmmse}
  Assume that the matrices $(A^{T}A + A^{T}\Sigma_{\varepsilon}A(A^{T}A)^{-1}\Sigma_{x^{\dag}}^{-1})^{-1}$ and $(A\Sigma_{x^{\dag}}A^T + \Sigma_{\varepsilon})^{-1}$ exist.
  Then the following are equivalent:
  \begin{enumerate}
  \item The best linear map for Lavrentiev regularization
    \begin{align*}
      W_{\textnormal{Lav}}:= (A^{T}A + A^{T}\Sigma_{\varepsilon}A(A^{T}A)^{-1}\Sigma_{x^{\dag}}^{-1})^{-1}A^{T}
    \end{align*}
    is equal to the linear map from the LLMSE estimator
    \begin{align*}
      W_{\textnormal{LMMSE}} := \Sigma_{x^{\dag}}A^{T}(A\Sigma_{x^{\dag}}A^{T} + \Sigma_{\varepsilon})^{-1};
    \end{align*}
  \item the noise covariance $\Sigma_{\varepsilon}$ leaves the kernel of $A^{T}$ invariant, i.e., if we denote by $P_{\ker(A^{T})}$ the projection onto the kernel of $A^{T}$, exactly if $A^{T}\Sigma_{\varepsilon}P_{\ker(A^{T})}=0$.
  \end{enumerate}
\end{theorem}
\begin{proof}
  We equate $W_{\textnormal{Lav}}$ and $W_{\textnormal{LMMSE}}$ and bring both inverse to the other sides to arrive at 
  \begin{align*}
    W_{\textnormal{Lav}} & = W_{\textnormal{LMMSE}}\\
    \iff & A^{T}(A\Sigma_{x^{\dag}}A^{T} + \Sigma_{\varepsilon}) = (A^{T}A + A^{T}\Sigma_{\varepsilon}A(A^{T}A)^{-1}\Sigma_{x^{\dag}}^{-1})\Sigma_{x^{\dag}}A^{T}\\
    \iff & A^{T}\Sigma_{\varepsilon} = A^{T}\Sigma_{\varepsilon}A(A^{T}A)^{-1}A^{T} \\
    \iff & A^{T}\Sigma_{\varepsilon}(I - A(A^{T}A)^{-1}A^{T}) = 0
  \end{align*}
  which is exactly the stated result since $I - A(A^{T}A)^{-1}A^{T} = P_{\ker(A^T)}$.
\end{proof}

As a consequence we may state:
\begin{quote}
  Learning how to regularize without also learning the noise weight is inferior to learning the noise weight and the regularizer as soon as $A^{T}\Sigma_{\varepsilon}P_{\ker(A^{T})}\neq 0$.
\end{quote}

Here a few special cases, when the condition $A^{T}\Sigma_{\varepsilon}P_{\ker(A^T)}$ is fulfilled:
\begin{itemize}
\item When $\ker(A) = \left\{ 0 \right\}$. Since we assume $m\geq n$ here in general (since we assume that $A^{T}A\in\RR^{n\times n}$ is invertible), this is fulfilled for invertible $A\in\RR^{n\times n}$.
\item When the noise covariance fulfills $\Sigma_{\varepsilon} = \sigma^{2}I$, i.e., when the noise is i.i.d.\ normally distributed with mean zero (since then $\Sigma_{\varepsilon}$ leaves every subspace invariant). This situation is often observed in numerical experiments when noise is artificially added by 
\begin{align*}
y = Ax^{\dag} + \varepsilon,\quad \varepsilon\sim \mathcal{N}(0,\sigma^{2}I),
\end{align*}
i.e., in code something like  \texttt{ydelta = y + sigma*randn(n)}.
\item More generally, whenever $\Sigma_{\varepsilon} = \sigma^{2}I + \tau^{2}AA^{T}$ which we get when we use the noise model 
\begin{align*}
y = A(x + \varepsilon_{x}) + \varepsilon_{y},\quad \varepsilon_{x}\sim \mathcal{N}(0,\tau^{2}I),\quad \varepsilon_{y}\sim \mathcal{N}(0,\sigma^{2}I).
\end{align*}
\end{itemize}

Even more colloquially we may state our result as
\begin{quote}
  If the noise is not too simple and you don't learn the noise weight, you leave something on the table.
\end{quote}

\section{Numerical experiments}
\label{sec:numerical-experiments}

In this section, we will describe experiments to investigate the different optimal regularizers that have been derived in the previous sections. The code to reproduce the numerical experiments can be found at \url{https://github.com/dirloren/learned-regularization}.

\subsection{Deconvolution of plateau functions under structured noise}

In our first experiment we consider a discrete deconvolution
problem. The training data is generated consisting of 50\,000 versions of vectors $x^{\dag}\in\RR^{n}$ for $n=200$ as follows: We view $x$ as a discretized function $x:[0,1]\to\RR$ and generate these by:
\begin{itemize}
\item Choose an integer $k$ between 2 and 5 uniformly at random
\item Set $x(t) = \sum\limits_{i=1}^{k} (a_{i}^{2}+0.01)\chi_{[c_{i}-b_{i},c_{i}+b_{i}]}(t)$ where the $a_{i}\sim \mathcal{N}(0,1)$  i.i.d., $b_{i}\sim\Unif([0,1])$ i.i.d., and $c_{i}\sim \Unif([0,0.15])$ i.i.d.
\end{itemize}
The mean and covariance of $x^{\dag}$ is approximated by the empirical mean and covariance.

The operator $A\in\RR^{m\times n}$ represents a convolution with a hat
function with width 30 (normalized to sum to one) and zero extension of $x^{\dag}$ so that we obtain $m=259$. To obtain the measurements $y$ we add noise $\varepsilon$ with $\varepsilon\sim \mathcal{N}(0,\diag(\sigma_{i}^{2}))$ where $\sigma_{i}$ decays linearly with $i$ from $\sigma_{1} = 10^{-2}$ to $\sigma_{m} = 5\cdot10^{-4}$ (i.e., there is larger noise for small $i$ and small noise for large $i$).
An example for the data $x^{\dag}$ and the respective data $y = Ax^{\dag} + \varepsilon$ as well as the empirical data mean is shown in Figure~\ref{fig:deconv-data-mean-examples}. The empirical data covariance is shown in Figure~\ref{fig:deconv-data-covariance}.

\begin{figure}[htb]
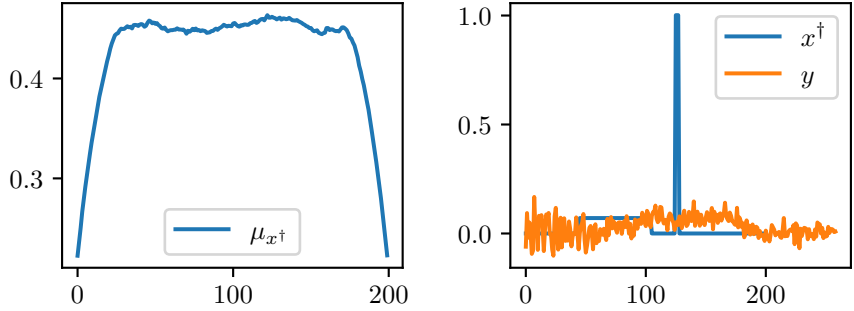
\centering
   \input{figures/deconv-data-mean.pgf}~\input{figures/deconv-data-examples.pgf}
  \caption{Left: The empirical mean of the data of experiment 1. Right: One sample of the data $x^\dag$ and the corresponding $y = Ax^{\dag} + \varepsilon$.}
  \label{fig:deconv-data-mean-examples}
\end{figure}

\begin{figure}[htb]\centering
 \includegraphics{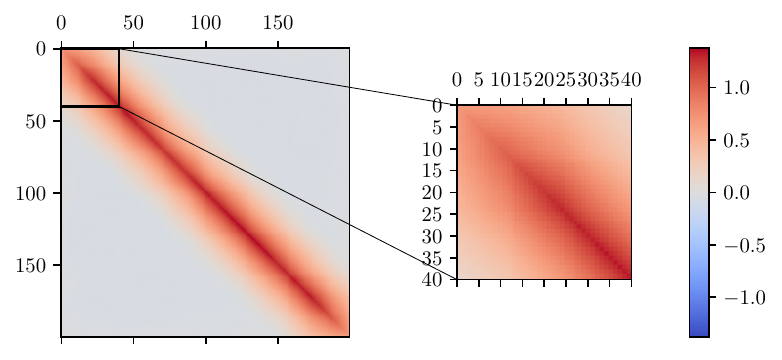}
  \caption{The empirical covariance matrix $\Sigma_{x^{\dag}}$ of the data $x^{\dag}$.}
  \label{fig:deconv-data-covariance}
\end{figure}

We then compute the LMMSE estimator map (which is equal to all the
learned regularization methods that also learn the noise weight) and
the best Lavrentiev and quadratic regularization (without noise
weight) according to the results from
Theorem~\ref{thm:optimal-lavrentiev} and
Theorem~\ref{thm:optimal-quadratic}, respectively. Both methods
compute a square matrix $M\in\RR^{n\times n}$ while the $M$ for best
quadratic regularization is symmetric by construction. Both maps are
shown in Figure~\ref{fig:deconv-opt-M-lav-quad}. The best $M$ for
Lavrentiev regularization is notably not symmetric and the relative
norm of the skew-symmetric part, i.e. $\tfrac{\tfrac12\norm{M -
    M^{T}}_{\textnormal{Fro}}}{\norm{M}_{\textnormal{Fro}}}$, is $0.59$.
The best $M$ for quadratic regularization is symmetric by construction, but not necessarily positve definite. In fact, the smallest eigenvalue is about $-0.19$ (and this does not seem to be a numerical issue as this number is stable with respect to the number of samples $x^{\dag}$ we use).

\begin{figure}[htb]
  \centering
  \includegraphics[width=0.7\textwidth]{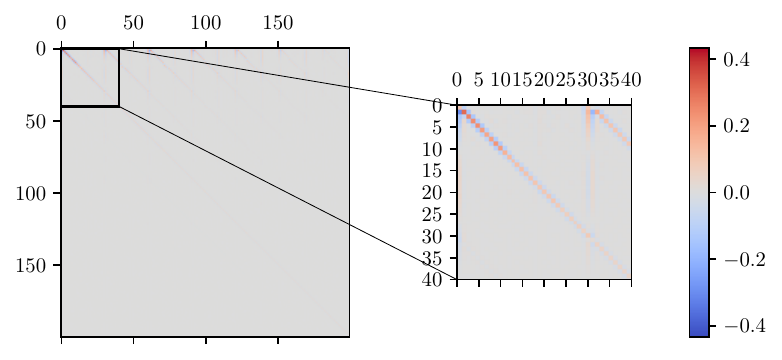}\\
  \includegraphics[width=0.7\textwidth]{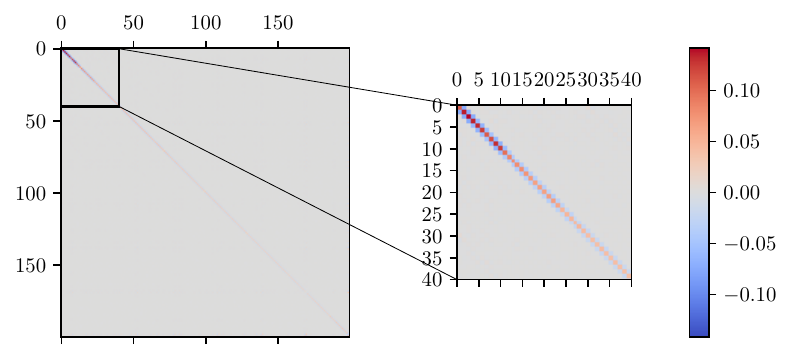}
  \caption{First row: Plot of the optimal matrix $M$ for Lavrentiev regularization. Bottom row: Plot of the optimal matrix $M$ for quadratic regularization. The right column shows a zoom on the top left $40\times 40$ block of the matrices on the left.}
  \label{fig:deconv-opt-M-lav-quad}
\end{figure}

Finally, we create test data consisting of 20\,000 samples from the
same distributions and compute the empirical losses over the test
data. The resulting values of the losses are shown in
Table~\ref{tab:losses}. We do not show the losses
$\mathcal{R}_{\textnormal{Tikh}(\Omega)}$,
$\mathcal{R}_{\textnormal{Lav}(\Omega)}$, and
$\mathcal{R}_{\textnormal{Quad}(\Omega)}$ as they are all equal to
$\mathcal{R}_{\textnormal{Aff}}$. It can be seen that the losses in
Table~\ref{tab:losses} are indeed all different, which empirically proves that the inequalities $\mathcal{R}_{\textnormal{Aff}}\leq \mathcal{R}_{\textnormal{Lav}} \leq \mathcal{R}_{\textnormal{Quad}}$ can be strict in practical examples.

\begin{table}[htb]
  \centering
  \begin{tabular}{lr}\toprule
    $\mathcal{R}_{\textnormal{Aff}}$& 23.12\\
    $\mathcal{R}_{\textnormal{Lav}}$& 23.23\\
    $\mathcal{R}_{\textnormal{Quad}}$ & 23.50\\\bottomrule
  \end{tabular}
  \caption{Empirical test losses for the LLMMSE (best affine), the best Lavrentiev regularization and the best quadratic regularization.}
  \label{tab:losses}
\end{table}

\subsection{Dereverberation of speech signals under simulated wind noise}

In our second experiment we use speech data from the IEEE-Harvard Corpus \cite{loizou2007speech}. This corpus includes 720 sentences spoken by male individuals, sampled at 16 kHz. We use the same split of the 720 overall signals into training (504 signals), validation (108 signals), and test (108 signals) data as was introduced in \cite{brauer2018primal} and used further in \cite{brauer2019learning, brauer2023asymptotic, brauer2024learning}. First, we normalize each full signal by dividing it by its maximum absolute value, such that each normalized signal has entries in the range $[-1, 1]$. Then, we split each normalized signal into non-overlapping frames of length 1\,000 resulting in 21\,147 frames for training and 4\,601 frames for testing (note that we do not make use of the validation data split in this experiment). Finally, we downsample each frame by a factor of two resulting in frames $x^\dag_i$ of length $n=500$.

\medskip

Regarding the forward operator, we consider another discrete deconvolution problem. In this case, $A\in\RR^{(2n - 1) \times n}$ is a reverberation matrix that models a sequence of decaying echoes. The associated convolution kernel $v\in\RR^n$ is defined via $v_1 \coloneqq 1$, $v_{i \cdot 50} \coloneqq 0.8^i$ for $i\in\{1,\dots,10\}$, and $v_j \coloneqq 0$ elsewhere. As in our first experiment, $A$ represents a full convolution using zero extension of the clean signal $x^\dag$ on both sides.

\medskip

Unlike the aforementioned earlier work, we do not use quantization noise \cite{brauer2018primal, brauer2019learning, brauer2023asymptotic} or i.i.d.\ standard normally distributed noise \cite{brauer2024learning} here. Quantization noise depends on the ground truth signal $x^\dagger$, which does not comply with our noise model. Moreover, since in this particular experiment we aim for noise that does not satisfy the condition from Theorem~\ref{thm:lavrentiev_equals_lmmse}, i.i.d.\ standard normal noise is also not considered. Instead, we generate noisy signals $y_i = Ax^\dag_i + \eta w_i \in \RR^{2n - 1}$ where $w_i$ is randomly generated wind noise and $\eta \in \{0.1, 0.2, 0.3, 0.4, 0.5\}$ is a factor controlling the noise level.

\begin{figure}[htb]
    \centering
    
    \begin{subfigure}{0.45\textwidth}
        \centering
        \includegraphics[width=\linewidth]{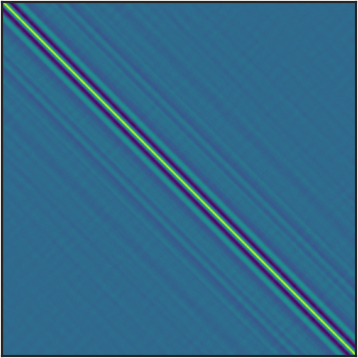}
        \caption{$\hat\Sigma_{x^\dag}$}
        \label{fig:sigma_x}
    \end{subfigure}
    \hfill
    \begin{subfigure}{0.45\textwidth}
        \centering
        \includegraphics[width=\linewidth]{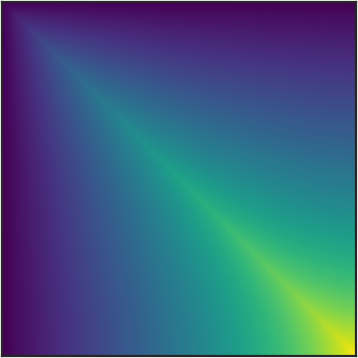}
        \caption{$\hat\Sigma_\varepsilon$}
        \label{fig:sigma_epsilon}
    \end{subfigure}
    
    \medskip
    
    \begin{subfigure}{\textwidth}
        \centering
        \includegraphics[width=\linewidth]{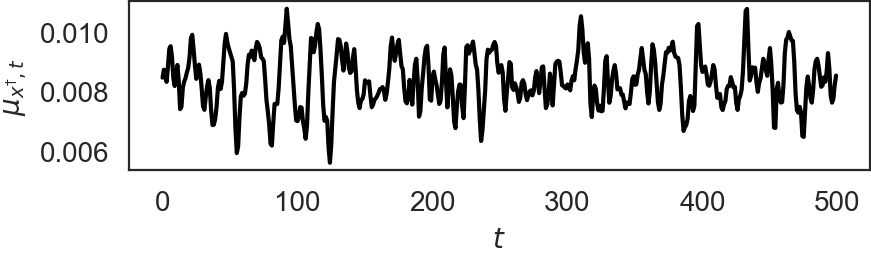}
        \caption{$\mu_x$}
        \label{fig:mu_x}
    \end{subfigure}
    
    \caption{Caption}
    \label{fig:placeholder}
\end{figure}

\medskip

Each instance of the wind noise $w \in \RR^{2n-1}$ is generated as follows: First, we sample an i.i.d.\ Gaussian noise vector $\gamma\in\RR^{2n-1}$, then cumulate $\beta_t \coloneqq \gamma_1 + \dots + \gamma_t$ for $t\in\{1,\dots,2n-1\}$ and set $b \coloneqq \beta / ||\beta||_\infty$ to get normalized Brownian noise $b\in \RR^{2n-1}$. Second, we apply a low-pass filter with a cutoff frequency of 3 kHz to $b$ to obtain $b_{\text{LP}}\in \RR^{2n-1}$. Third, we create a bursty amplitude envelope $e\in \RR^{2n-1}$. And finally, we modulate the filtered signal with the bursty envelope, add low-frequency sinusoidal modulations and small stochastic perturbations, namely
\begin{equation*}
    w_t = b_{\text{LP}, t} \cdot e_t \cdot (1 + \tfrac{1}{10}\sin (2\pi f_t \tfrac{t - 1}{8000} + \phi_t)) + \tfrac{\epsilon_t}{1000} \quad \text{for} \quad t \in \{1, \dots, 2n-1\}
\end{equation*}
with random $f_t \sim \Unif([0.1, 0.5])$, $\phi_t \sim \Unif([0, 2\pi])$ and $\epsilon_t \sim \mathcal{N}(0, 1)$. For a comprehensive overview on wind noise modeling we refer to \cite{nelke2016wind}.

\medskip

In addition to testing the performance of the theoretically optimal reconstruction maps for LMMSE estimation (cf. Corollary~\ref{cor:llmmse}), Lavrentiev regularization (cf. Theorem~\ref{thm:optimal-lavrentiev}) and quadratic regularization (cf. Theorem~\ref{thm:optimal-quadratic}) on speech signals corrupted by wind noise, we also learn respective reconstruction maps from the training data via gradient descent. This also enables us to incorporate Tikhonov regularization, for which we did not derive an optimal reconstruction map. In total, we learn four parameterized reconstruction maps per noise level, which are summarized in Table~\ref{tab:learning_objectives}. The first minimization problem seeks the optimal affine linear mapping.  The other three correspond to the maps without noise weight from Table~\ref{tab:regularizaton-methods}.

\begin{table}[htb]
    \centering
    \begin{tabular}{ll}
    \toprule
    Learned $\mathcal{R}_{\textnormal{Aff}}$  &  $\min_{W, b} \ \ \tfrac{1}{m n}\sum_{i=1}^{m} \norm{Wy_i + b - x^\dag_i}^2$ \\
    Learned $\mathcal{R}_{\textnormal{Lav}}$  &  $\min_{M, x_0} \ \ \tfrac{1}{m n}\sum_{i=1}^{m} \norm{(A^T A + M)^{-1}(A^T y_i + x_0) - x^\dag_i}^2$ \\
    Learned $\mathcal{R}_{\textnormal{Quad}}$ &  $\min_{L, x_0} \ \ \tfrac{1}{m n}\sum_{i=1}^{m} \norm{(A^T A + \tfrac12[L + L^T])^{-1}(A^T y_i + x_0) - x^\dag_i}^2$ \\
    Learned $\mathcal{R}_{\textnormal{Tikh}}$  &  $\min_{R, x_0} \ \ \tfrac{1}{m n}\sum_{i=1}^{m} \norm{(A^T A + R^T R)^{-1}(A^T y_i + x_0) - x^\dag_i}^2$ \\
    \bottomrule
    \end{tabular}
    \caption{Learning objectives for data-based training of reconstruction maps. Here, $m$ denotes the number of training examples and $n$ is the dimension of the ground truth data. In the first row, the variable dimensions are $W\in\RR^{n \times (2n-1)}$ and $b\in\RR^{2n-1}$. Apart from that, $M, L, R\in \RR^{n\times n}$ and $x_0\in\RR^n$ throughout.}
    \label{tab:learning_objectives}
\end{table}

\begin{table}[]
    \centering
    \begin{tabular}{lccccc}
    \toprule
    Noise level $\eta$ &        0.1 &        0.2 &        0.3 &        0.4 &        0.5 \\
    \midrule
    Optimal $\mathcal{R}_{\textnormal{Aff}}$ & 7.05e-05 & 1.98e-04 & 3.80e-04 & 6.03e-04 & 8.56e-04 \\
    Optimal $\mathcal{R}_{\textnormal{Lav}}$ & 7.28e-05 & 2.05e-04 & 3.93e-04 & 6.24e-04 & 8.84e-04 \\
    Optimal $\mathcal{R}_{\textnormal{Quad}}$ & 9.16e-05 & 2.31e-04 & 4.25e-04 & 6.62e-04 & 9.29e-04 \\
    \midrule
    Learned $\mathcal{R}_{\textnormal{Aff}}$ & 7.10e-05 & 1.99e-04 & 3.79e-04 & 6.03e-04 & 8.55e-04 \\
    Learned $\mathcal{R}_{\textnormal{Lav}}$ & 9.38e-05 & 2.54e-04 & 4.67e-04 & 7.11e-04 & 9.80e-04 \\
    Learned $\mathcal{R}_{\textnormal{Quad}}$ & 1.02e-04 & 2.68e-04 & 4.82e-04 & 7.29e-04 & 1.00e-03 \\
    Learned $\mathcal{R}_{\textnormal{Tikh}}$ & 1.06e-04 & 2.76e-04 & 4.86e-04 & 7.32e-04 & 1.01e-03 \\
    \bottomrule
    \end{tabular}
    \caption{Mean squared error of different reconstruction maps for different noise levels on the training data.}
    \label{tab:speech_mse_results_train}
\end{table}

\begin{table}[htb]
    \centering
    \begin{tabular}{lccccc}
    \toprule
    Noise level $\eta$ &        0.1 &        0.2 &        0.3 &        0.4 &        0.5 \\
    \midrule
    Optimal $\mathcal{R}_{\textnormal{Aff}}$  & 7.51e-05 & 2.11e-04 & 4.05e-04 & 6.41e-04 & 9.06e-04 \\
    Optimal $\mathcal{R}_{\textnormal{Lav}}$  & 7.35e-05 & 2.05e-04 & 3.93e-04 & 6.22e-04 & 8.81e-04 \\
    Optimal $\mathcal{R}_{\textnormal{Quad}}$ & 8.92e-05 & 2.25e-04 & 4.17e-04 & 6.48e-04 & 9.08e-04 \\
    \midrule
    Learned $\mathcal{R}_{\textnormal{Aff}}$  & 7.50e-05 & 2.10e-04 & 4.01e-04 & 6.36e-04 & 9.05e-04 \\
    Learned $\mathcal{R}_{\textnormal{Lav}}$  & 9.48e-05 & 2.53e-04 & 4.66e-04 & 7.11e-04 & 9.86e-04 \\
    Learned $\mathcal{R}_{\textnormal{Quad}}$ & 1.01e-04 & 2.62e-04 & 4.72e-04 & 7.15e-04 & 9.88e-04 \\
    Learned $\mathcal{R}_{\textnormal{Tikh}}$  & 1.04e-04 & 2.69e-04 & 4.76e-04 & 7.18e-04 & 9.91e-04 \\
    \bottomrule
    \end{tabular}
    \caption{Mean squared error of different reconstruction maps for different noise levels on the test data. The reported errors in this table correspond exactly to those illustrated in Figure~\ref{tab:speech_mse_results}.}
    \label{tab:speech_mse_results}
\end{table}

\begin{figure}[htb]
    \centering
    \includegraphics{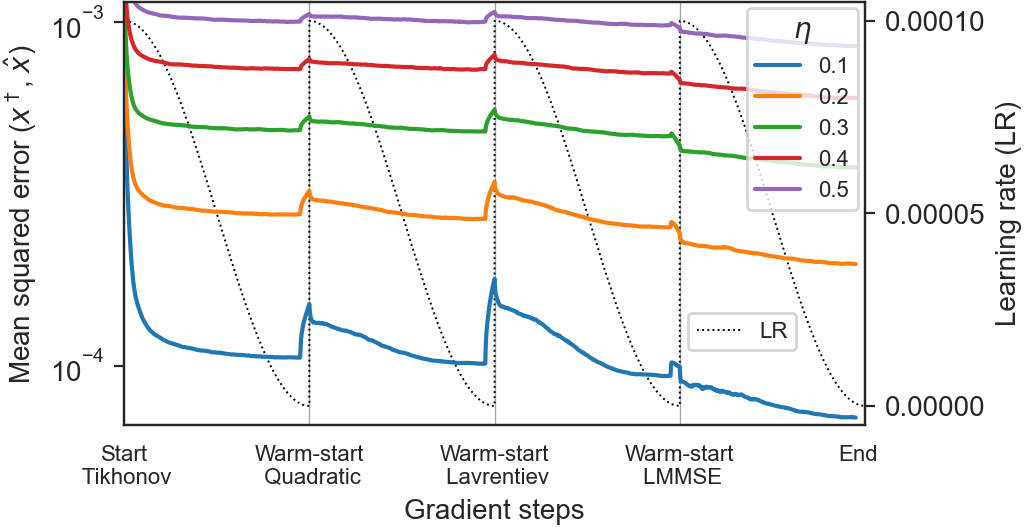}
    \caption{Illustration of the learning progress using a warm-start strategy in the transition from method to method. The different colors indicate different noise levels, as shown on the left vertical axis. While optimization variables are carried over during each warm-start, the learning rate decay is the same for all four methods, as indicated by the dotted line. All error curves are smoothed using a moving average of length 6610, corresponding to 5\% of the length of one epoch.}
    \label{fig:wind_experiment_learning_curves}
\end{figure}

\medskip

In each case, stochastic gradient descent is carried out using automatic differentiation in \textsc{TensorFlow}~\cite{tensorflow2015-whitepaper} and the Adam~\cite{kingma2015adam} optimizer with a cosine decay learning rate schedule with initial learning rate $10^{-4}$, a batch size of 32, and 200 epochs of training. Moreover, we apply the following warm-start strategy: For each noise level, we first learn $\mathcal{R}_{\textnormal{Tikh}}$. Then, we initialize the optimization variable for $\mathcal{R}_{\textnormal{Quad}}$ using the final Tikhonov solution, $L = R^TR$, carrying over the state of $x_0$. Similarly, we then transition to $\mathcal{R}_{\textnormal{Lav}}$ initializing $M = L$ with the optimal quadratic map and again carrying over $x_0$. The optimization variables for $\mathcal{R}_{\textnormal{Aff}}$ are finally initialized using $W = (A^T A + M)^{-1} A^T$ and $b = (A^T A + M)^{-1} M x_0$ using the final Lavrentiev iterates $M$ and $x_0$.

In Table~\ref{tab:speech_mse_results_train} we observe that in general the performance of the learned regularizers on the training data is close to the performance of the optimal ones that were computed with the empirical mean and covariance matrices. Comparing the numbers with the ones on the test data in Table~\ref{tab:speech_mse_results} we see that the generalization error (i.e. the difference between $\mathcal{R}^{\textnormal{test}}$ and $\mathcal{R}^{\textnormal{train}}$) is in general quite small.
However, we also observe that the risk for the theoretically best method $\mathcal{R}_{\textnormal{Aff}}$ is larger than $\mathcal{R}_{\textnormal{Lav}}$ when we use the formulae derived in the paper.
We suspect that the reason is that we used the empirical covariance and mean from the training set which lead to overfitting to this data. A possible explanation for the worse generalization of the best affine map is that it has more degrees of freedom that the best Lavrentiev and quadratic regularization ($mn+n$ vs. $n^{2}+n$ and we have $m>n$) and hence, may be more prone to overfitting the data in the finite training set.
For the learned methods we observe better generalization and also the theoretically predicted order of the methods $\mathcal{R}_{\textnormal{Aff}}\leq \mathcal{R}_{\textnormal{Lav}}\leq \mathcal{R}_{\textnormal{Quad}}\leq \mathcal{R}_{\textnormal{Tikh}}$ is fulfilled on both test and training data.

\section{Conclusion}
\label{sec:conclusion}
Our analysis of the regularization methods shows that there are performance gaps between Tikhonov,
Lavrentiev and quadratic regularization if the weight of their respective
data fidelity terms is not compatible with the covariance of the
noise. This underscores the importance of learning not only the prior
distribution, but also the noise model in data-driven
regularization. If, in turn, the noise model is known, then weighted
Tikhonov regularization is sufficient to recover the optimal affine
reconstruction found by Alberti et al.\
\cite{alberti2021learning}. Our numerical experiments confirm that, if
the noise model is not known or not used, it can be advantageous to
employ regularizers that are not positive semidefinite or even
asymmetric. Future research should investigate our statements if the
assumptions on uncorrelated, additive noise are relaxed, since our
experiments confirmed some of our analytic results even in this
scenario.
It remains open how the optimal Tikhonov regularizers $R$ (which are in general not unique) look like.
Moreover, it would be interesting to understand better which factors influence the differences in the optimal risks for the different methods.

\bibliographystyle{plain}
\bibliography{references.bib}

\begin{thebibliography}{10}

\bibitem{tensorflow2015-whitepaper}
Mart\'{i}n Abadi, Ashish Agarwal, Paul Barham, Eugene Brevdo, Zhifeng Chen,
  Craig Citro, Greg~S. Corrado, Andy Davis, Jeffrey Dean, Matthieu Devin,
  Sanjay Ghemawat, Ian Goodfellow, Andrew Harp, Geoffrey Irving, Michael Isard,
  Yangqing Jia, Rafal Jozefowicz, Lukasz Kaiser, Manjunath Kudlur, Josh
  Levenberg, Dandelion Man\'{e}, Rajat Monga, Sherry Moore, Derek Murray, Chris
  Olah, Mike Schuster, Jonathon Shlens, Benoit Steiner, Ilya Sutskever, Kunal
  Talwar, Paul Tucker, Vincent Vanhoucke, Vijay Vasudevan, Fernanda Vi\'{e}gas,
  Oriol Vinyals, Pete Warden, Martin Wattenberg, Martin Wicke, Yuan Yu, and
  Xiaoqiang Zheng.
\newblock {TensorFlow}: Large-scale machine learning on heterogeneous systems,
  2015.
\newblock Software available from tensorflow.org.

\bibitem{adler2017solving}
Jonas Adler and Ozan {\"O}ktem.
\newblock Solving ill-posed inverse problems using iterative deep neural
  networks.
\newblock {\em Inverse Problems}, 33(12):124007, 2017.

\bibitem{alberti2025learning}
Giovanni~S Alberti, Luca Ratti, Matteo Santacesaria, and Silvia Sciutto.
\newblock Learning a gaussian mixture for sparsity regularization in inverse
  problems.
\newblock {\em IMA Journal of Numerical Analysis}, page draf037, 2025.

\bibitem{alberti2021learning}
Giovanni~S Alberti, Ernesto~De Vito, Matti Lassas, Luca Ratti, and Matteo
  Santacesaria.
\newblock Learning the optimal tikhonov regularizer for inverse problems.
\newblock In A.~Beygelzimer, Y.~Dauphin, P.~Liang, and J.~Wortman Vaughan,
  editors, {\em Advances in Neural Information Processing Systems}, volume~34,
  pages 25205--25216, 2021.

\bibitem{argyrou2012tomographic}
Maria Argyrou, Dimitris Maintas, Charalampos Tsoumpas, and Efstathios
  Stiliaris.
\newblock Tomographic image reconstruction based on artificial neural network
  ({ANN}) techniques.
\newblock In {\em 2012 IEEE Nuclear Science Symposium and Medical Imaging
  Conference Record (NSS/MIC)}, pages 3324--3327. IEEE, 2012.

\bibitem{arridge2019solving}
Simon Arridge, Peter Maass, Ozan {\"O}ktem, and Carola-Bibiane Sch{\"o}nlieb.
\newblock Solving inverse problems using data-driven models.
\newblock {\em Acta Numerica}, 28:1--174, 2019.

\bibitem{bhatia2013matrix}
Rajendra Bhatia.
\newblock {\em Matrix analysis}, volume 169.
\newblock Springer Science \& Business Media, 2013.

\bibitem{brauer2024learning}
Christoph Brauer, Niklas Breustedt, Timo de~Wolff, and Dirk~A Lorenz.
\newblock Learning variational models with unrolling and bilevel optimization.
\newblock {\em Analysis and Applications}, 22(03):569--617, 2024.

\bibitem{brauer2018primal}
Christoph Brauer and Dirk Lorenz.
\newblock Primal-dual residual networks.
\newblock {\em arXiv preprint arXiv:1806.05823}, 2018.

\bibitem{brauer2023asymptotic}
Christoph Brauer and Dirk~A Lorenz.
\newblock Asymptotic analysis and truncated backpropagation for the unrolled
  primal-dual algorithm.
\newblock In {\em 2023 31st European Signal Processing Conference (EUSIPCO)},
  pages 860--864. IEEE, 2023.

\bibitem{brauer2019learning}
Christoph Brauer, Ziyue Zhao, Dirk Lorenz, and Tim Fingscheidt.
\newblock Learning to dequantize speech signals by primal-dual networks: an
  approach for acoustic sensor networks.
\newblock In {\em ICASSP 2019-2019 IEEE International Conference on Acoustics,
  Speech and Signal Processing (ICASSP)}, pages 7000--7004. IEEE, 2019.

\bibitem{goujon2023neural}
Alexis Goujon, Sebastian Neumayer, Pakshal Bohra, Stanislas Ducotterd, and
  Michael Unser.
\newblock A neural-network-based convex regularizer for inverse problems.
\newblock {\em IEEE Transactions on Computational Imaging}, 9:781--795, 2023.

\bibitem{hamarik2008extrapolation}
Uno H{\"a}marik, Reimo Palm, and Toomas Raus.
\newblock Extrapolation of tikhonov and lavrentiev regularization methods.
\newblock In {\em Journal of Physics: Conference Series, 6TH INTERNATIONAL
  CONFERENCE ON INVERSE PROBLEMS IN ENGINEERING: THEORY AND PRACTICE}, volume
  135, page 012048. IOP Publishing, 2008.

\bibitem{hammernik2018learning}
Kerstin Hammernik, Teresa Klatzer, Erich Kobler, Michael~P Recht, Daniel~K
  Sodickson, Thomas Pock, and Florian Knoll.
\newblock Learning a variational network for reconstruction of accelerated mri
  data.
\newblock {\em Magnetic resonance in medicine}, 79(6):3055--3071, 2018.

\bibitem{hauptmann2018model}
Andreas Hauptmann, Felix Lucka, Marta Betcke, Nam Huynh, Jonas Adler, Ben Cox,
  Paul Beard, Sebastien Ourselin, and Simon Arridge.
\newblock Model-based learning for accelerated, limited-view 3-d photoacoustic
  tomography.
\newblock {\em IEEE transactions on medical imaging}, 37(6):1382--1393, 2018.

\bibitem{kay1993fundamentals}
Steven~M Kay.
\newblock {\em Fundamentals of Statistical Signal Processing: Estimation
  Theory}.
\newblock Prentice Hall, 1993.

\bibitem{kingma2015adam}
Diederik~P. Kingma and Jimmy Ba.
\newblock Adam: A method for stochastic optimization.
\newblock In Yoshua Bengio and Yann LeCun, editors, {\em ICLR (Poster)}, 2015.

\bibitem{kobak2020optimal}
Dmitry Kobak, Jonathan Lomond, and Benoit Sanchez.
\newblock The optimal ridge penalty for real-world high-dimensional data can be
  zero or negative due to the implicit ridge regularization.
\newblock {\em Journal of Machine Learning Research}, 21(169):1--16, 2020.

\bibitem{kobler2021total}
Erich Kobler, Alexander Effland, Karl Kunisch, and Thomas Pock.
\newblock Total deep variation: A stable regularization method for inverse
  problems.
\newblock {\em IEEE transactions on pattern analysis and machine intelligence},
  44(12):9163--9180, 2021.

\bibitem{lancaster1995algebraic}
Peter Lancaster and Leiba Rodman.
\newblock {\em Algebraic riccati equations}.
\newblock Clarendon press, 1995.

\bibitem{li2020nett}
Housen Li, Johannes Schwab, Stephan Antholzer, and Markus Haltmeier.
\newblock Nett: Solving inverse problems with deep neural networks.
\newblock {\em Inverse Problems}, 36(6):065005, 2020.

\bibitem{loizou2007speech}
Philipos~C Loizou.
\newblock {\em Speech enhancement: theory and practice}.
\newblock CRC press, 2007.

\bibitem{lorenz2013necessary}
Dirk Lorenz and Nadja Worliczek.
\newblock Necessary conditions for variational regularization schemes.
\newblock {\em Inverse Problems}, 29(7):075016, 2013.

\bibitem{lunz2018adversarial}
Sebastian Lunz, Ozan {\"O}ktem, and Carola-Bibiane Sch{\"o}nlieb.
\newblock Adversarial regularizers in inverse problems.
\newblock {\em Advances in neural information processing systems}, 31, 2018.

\bibitem{mukherjee2021end}
Subhadip Mukherjee, Marcello Carioni, Ozan {\"O}ktem, and Carola-Bibiane
  Sch{\"o}nlieb.
\newblock End-to-end reconstruction meets data-driven regularization for
  inverse problems.
\newblock {\em Advances in Neural Information Processing Systems},
  34:21413--21425, 2021.

\bibitem{mukherjee2020learned}
Subhadip Mukherjee, S{\"o}ren Dittmer, Zakhar Shumaylov, Sebastian Lunz, Ozan
  {\"O}ktem, and Carola-Bibiane Sch{\"o}nlieb.
\newblock Learned convex regularizers for inverse problems.
\newblock arXiv preprint arXiv:2008.02839, 2020.

\bibitem{mukherjee2021learning}
Subhadip Mukherjee, Carola-Bibiane Sch{\"o}nlieb, and Martin Burger.
\newblock Learning convex regularizers satisfying the variational source
  condition for inverse problems.
\newblock {\em arXiv preprint arXiv:2110.12520}, 2021.

\bibitem{nelke2016wind}
Christoph~Matthias Nelke.
\newblock {\em Wind noise reduction: signal processing concepts}.
\newblock Wissenschaftsverlag Mainz, 2016.

\bibitem{semenova2010lavrentiev}
Evgeniya~V Semenova.
\newblock Lavrentiev regularization and balancing principle for solving
  ill-posed problems with monotone operators.
\newblock {\em Computational Methods in Applied Mathematics}, 10(4):444--454,
  2010.

\bibitem{tautenhahn2002method}
U~Tautenhahn.
\newblock On the method of lavrentiev regularization for nonlinear ill-posed
  problems.
\newblock {\em Inverse Problems}, 18(1):191, 2002.

\bibitem{tsigler2023benign}
Alexander Tsigler and Peter~L Bartlett.
\newblock Benign overfitting in ridge regression.
\newblock {\em Journal of Machine Learning Research}, 24(123):1--76, 2023.

\bibitem{vapnik1991principles}
Vladimir Vapnik.
\newblock Principles of risk minimization for learning theory.
\newblock {\em Advances in neural information processing systems}, 4, 1991.

\bibitem{zhang2025learning}
Yasi Zhang and Oscar Leong.
\newblock Learning difference-of-convex regularizers for inverse problems: A
  flexible framework with theoretical guarantees.
\newblock {\em arXiv preprint arXiv:2502.00240}, 2025.

\end{thebibliography}

\end{document}